\documentclass[12pt,reqno]{amsart} 
\usepackage{mathtools}

\mathtoolsset{showonlyrefs}
\usepackage[hmarginratio={1:1},scale=0.8]{geometry} 
\usepackage{enumerate} 
\usepackage{amssymb} 
\usepackage{cite,color} 
\usepackage{bm} \numberwithin{equation}{section}

\usepackage{amsfonts} 

\newcommand{\rn}{\mathbb R^n}

\newcommand{\sn}{ {S^{n-1}}}
\newcommand{\s}{ {S^{1}}}

\newtheorem{lemma}{Lemma}[section] 
\newtheorem{theorem}[lemma]{Theorem} 
\newtheorem{conjecture}[lemma]{Conjecture}

\newtheorem{defi}[lemma]{Definition} 
\newtheorem{coro}[lemma]{Corollary} 
 
\newtheorem{prop}[lemma]{Proposition} 
\newtheorem{remark}[lemma]{Remark} 
\title{On the planar Gaussian-Minkowski problem} 
\author[S. Chen]{Shibing Chen}
\address{School of Mathematical Sciences, University of Science and Technology of China,Hefei, 230026, China}
\email{chenshib@ustc.edu.cn}

\author[S. Hu]{Shengnan Hu}
\address{School of Mathematical Sciences, University of Science and Technology of China,Hefei, 230026, China}
\email{helenhsn@mail.ustc.edu.cn}

\author[W. Liu]{Weiru Liu}
\address{School of Mathematical Sciences, University of Science and Technology of China,Hefei, 230026, China}
\email{lwr19997@mail.ustc.edu.cn}

\author[Y. Zhao]{Yiming Zhao}
\address{Department of Mathematics, Syracuse University, Syracuse, NY 13244, USA}
\email{yzhao197@syr.edu}
\thanks{Research of Chen was supported by National Key R and D program of China 2022YFA1005400, NSFC No. 12225111, NSFC No. 12071482 and
NSFC No. 12141105.}
\thanks{Research of Zhao was supported, in part, by US NSF Grant DMS--2132330.}

\begin{document}
\begin{abstract} 
The current work focuses on the Gaussian-Minkowski problem in dimension 2. In particular, we show that if the Gaussian surface area measure is proportional to the spherical Lebesgue measure, then the corresponding convex body has to be a centered disk. As an application, this ``uniqueness'' result is used to prove the existence of smooth small solutions to the Gaussian-Minkowski problem via a degree-theoretic approach.
\end{abstract}
\date{\today}

\maketitle


\section{Introduction}

The study of Minkowski-type problems and their associated isoperimetric inequalities is at the core of convex geometric analysis. This line of research can be traced back to the seminal work \cite{minkowski2} of Minkowski and the celebrated Brunn-Minkowski inequality due to Brunn and Minkowski in the 19th century. The comprehensive survey \cite{MR1898210} by Gardner explores the many connections between the Brunn-Minkowski inequality (and its various generalizations) and other isoperimetric-type inequalities, both geometric and analytic. The Brunn-Minkowski inequality is strongly connected to the classical Minkowski problem which asks for the existence and uniqueness of a convex body whose \emph{surface area measure} is prescribed by a given Borel measure $\mu$ on $\sn$. Here, by convex body we mean a compact convex subset of $\rn$ with nonempty interior, and if $K$ is a convex body, its surface area measure is given by
\begin{equation}
	S_K(\eta)=\mathcal{H}^{n-1}(\nu_K^{-1}(\eta)),
\end{equation}
for each Borel set $\eta\subset \sn$ where $\nu_K^{-1}$ is the inverse Gauss map. Surface area measure can be viewed as the differential of the volume functional.  The Minkowski problem is equivalent to the following PDE on $\sn$:
\begin{equation}
	\det(\nabla^2_{\sn} h+hI)=\mu,
\end{equation}
and motivated much of the development of Monge-Amp\`{e}re equations in the last century. See the works of Minkowski \cite{MR1511220}, Aleksandrov \cite{MR0001597}, Cheng-Yau \cite{MR0423267}, Pogorelov\cite{MR0478079}, and Caffarelli \cite{MR1005611, MR1038359,MR1038360}. In particular, in differential geometry, it is the problem of prescribing Gauss curvature. To see intuitively why the Brunn-Minkowski inequality and the Minkowski problem are naturally intertwined, note that the Brunn-Minkowski inequality states roughly that the $n$-th root of the volume functional $V^{\frac{1}{n}}$ is concave, whereas the Minkowski problem studies surface area measures---the ``derivative'' of volume functional. 

Motivated by the success of the classical Brunn-Minkowski theory, in the last three decades, there have been some crucial types of variants of the classical Minkowski problem---all of them involving prescribing certain geometric measures generated by ``differentiating'' geometric invariants in ways similar to the one leading to surface area measure. These Minkowski problems can be viewed as the problem of prescribing different curvature functions in the smooth case, and lead to many new (and challenging) Monge-Amp\`{e}re equations. Some of the most prominent Minkowski-type problems include the $L_p$ Minkowski problem (see \cite{MR1231704, MR2132298,MR2254308}), the logarithmic Minkowski problem (see \cite{MR3037788}), and the dual Minkowski problem (see \cite{MR3573332}). More details will be provided. 

It is natural to attempt to migrate this theory to other measurable spaces. Of particular interest is $\rn$ when equipped with the Gaussian probability measure $\gamma_n$ given by 
\begin{equation}
	\gamma_n(E)=\frac{1}{(2\pi)^\frac{n}{2}} \int_{E}e^{-\frac{|x|^2}{2}}dx.
\end{equation} 
Note that a very distinct feature of Gaussian probability measure, when compared to Lebesgue measure, is that it is not uniform, or even homogeneous. This makes it conceptually extremely challenging to even believe the validity of a Brunn-Minkowski inequality in $(\rn, \gamma_n)$---note that the exponent $\frac{1}{n}$ in the Brunn-Minkowski inequality corresponds to how the volume scales when a convex body is rescaled. Instead, the isoperimetric-type inequality in the Gaussian probability space is the Erhard inequality \eqref{eq 3.28.4} which does not make use of homogeneity. It is quite surprising to see that a dimensional Brunn-Minkowski inequality actually holds in $(\rn, \gamma_n)$ for origin-symmetric convex bodies. This was initially conjectured by Gardner-Zvavitch \cite{MR2657682}, with important contribution by Kolesnikov-Livshyts \cite{MR4238914} followed by a recent confirmation by Eskenazis-Moschidis \cite{alex2020dimensional}. Gardner-Zvavitch \cite{MR2657682} observed that this inequality neither implies nor is implied by Ehrhard's inequality. The dimensional Brunn-Minkowski inequality is also linked with the conjectured log-Brunn-Minkowski inequality (planar case established in \cite{MR2964630})---an inequality in the Lebesgue measure space but with different addition---following a result by Livshyts-Marsiglietti-Nayar-Zvavitch \cite{MR3710641}, which was very recently extended in \cite{hosle2020l_p}. We also would like to mention the work of Borell \cite{MR399402}.

Motivated by the rich theory regarding isoperimetric inequalities in $(\rn, \gamma_n)$, Huang, Xi and the last author \cite{MR4252759} studied the Minkowski problem in $(\rn, \gamma_n)$. Let $K$ be a convex body in $\rn$ that contains the origin as an interior point. The Gaussian surface area of $K$, denoted by $S_{\gamma_n, K}$, is the unique Borel measure that satisfies
\begin{equation}
	\lim_{t\rightarrow 0} \frac{\gamma_n(K+tL)-\gamma_n(K)}{t} = \int_{\sn } h_L dS_{\gamma_n, K}
\end{equation}
for each convex body $L$ in $\rn$ that contains the origin as an interior point. Here $h_L$ is the support function of $L$, see \eqref{eq 3.28.5}. A more explicit formula for $S_{\gamma_n, K}$ is given in \eqref{eq 3.28.6}. 

\textbf{The Gaussian Minkowski problem.} Given a finite Borel measure $\mu$, what are the necessary and sufficient conditions on $\mu$ so that there exists a convex body $K$ with $o\in \text{int}\,K$ such that
\begin{equation}
\label{eq local 00005}
	\mu  = S_{\gamma_n, K}?
\end{equation}
If $K$ exists, to what extent is it unique?

When the given measure $\mu$ has a density $d\mu = fdv$, \eqref{eq local 00005} is equivalent to the following PDE on $\sn$:
\begin{equation}
\label{eq 3.28.7}
	\frac{1}{(2\pi)^\frac{n}{2}} e^{-\frac{|\nabla h|^2+h^2}{2}}\det (\nabla^2 h+hI) = f.
\end{equation}

Due to the unique properties of the Gaussian density, in \cite{MR4252759}, results regarding the Gaussian Minkowski problem are restricted to the subset $\gamma_n(K)\geq 1/2$. This is where the Erhard inequality readily implies the uniqueness part of the Gaussian Minkowski problem: if $ S_{\gamma_n, K}=S_{\gamma_n, L}$ and $ \gamma_n(K), \gamma_n(L)\geq 1/2$, then $K=L$. See Theorem 1.1 in \cite{MR4252759}. We point out that without this restriction, the Erhard inequality is of no use for the uniqueness part of Gaussian Minkowski problem. In particular, when $\gamma_n(K)<\frac{1}{2}$, it is in fact possible that \eqref{eq local 00005} more than one solutions. This conceptually makes the existence part of the Gaussian Minkowski problem more challenging as the possibility of multiple solutions often implies increasing difficulties in providing $C^0$ estimates required in the solution of a Minkowski problem. 

In the current work, we study the existence and the uniqueness part of the Gaussian Minkowski problem in dimension 2 \footnote{Technically speaking, the Gaussian surface area in dimension 2 should be better referred to as the Gaussian perimeter measure. However, we shall keep its original name, for consistence with the case in higher dimensions.}, without the restriction that $\gamma_2(K)\geq 1/2$.

To fully appreciate the unique challenge in the Gaussian Minkowski problem, we provide the background in other \emph{relatively} more well-studied Minkowski-type problems for comparison.

There are two major variants of the classical Brunn-Minkowski theory. 

One is the $L_p$ Brunn-Minkowski theory initiated by the two landmark papers \cite{MR1231704, MR1378681} by Lutwak in the early 1990s where he defined the \emph{$L_p$ surface area measure} fundamental in the now fruitful $L_p$ Brunn-Minkowski theory  central in modern convex geometric analysis. It is crucial to point out that such extension is highly nontrivial and often requires new techniques. See, for example, \cite{MR2254308, MR2132298,MR2067123,MR3366854,MR3356071,MR3415694,MR3148545,MR2652209,MR2680490,MR1901250,MR2019226,MR2729006,MR1987375} for a (not even close to exhaustive) list of works in the $L_p$ Brunn-Minkowski theory. In particular, the theory becomes \emph{significantly} harder when $p<1$. These include the critical centro-affine case $p=-n$ and the logarithmic case $q=0$. Isoperimetric inequalities and Minkowski problems in neither cases have been fully addressed. In particular, the \emph{log Minkowski problem} (for the \emph{cone volume measure}) has not yet been fully solved. See, for example, Bianchi-B\"or\"oczky-Colesanti-Yang \cite{MR3872853}, Chou-Wang \cite{MR2254308}, Guang-Li-Wang \cite{https://doi.org/10.48550/arxiv.2203.05099}, Zhu \cite{MR3356071, MR3228445} among many other works. In fact, the $p=0$ case harbors the \emph{log Brunn-Minkowski conjecture} (see, for example, B\"or\"oczky-LYZ \cite{MR2964630})---arguably the most crucial conjecture in convex geometric analysis in the past decade. The log Brunn-Minkowski conjecture has been verified in dimension 2 and in various special classes of convex bodies. See, for example, Chen-Huang-Li-Liu \cite{MR4088419}, Colesanti-Livshyts-Marsiglietti \cite{MR3653949},  Kolesnikov-Livshyts \cite{MR4485961}, Kolesnikov-Milman \cite{MR4438690},  Milman \cite{https://doi.org/10.48550/arxiv.2104.12408}, Putterman \cite{MR4220744}, Saroglou \cite{MR3370038}. If proven correct, it is much stronger than the classical Brunn-Minkowski inequality.

The other is the dual Brunn-Minkowski theory initiated by Lutwak in the 1970s.  Compared to the classical theory which focuses more on projections and boundary shapes of convex bodies, the dual Brunn-Minkowski theory focuses more on intersections and interior properties of convex bodies. This explains the crucial role that the dual theory played in the solution of the well-known and the then long-standing \emph{Busemann-Petty problem} in the 1990s. See, for example, \cite{MR1298719, MR1689343, MR963487, MR1689339}.
The counterparts for the quermassintegrals in the dual theory are the \emph{dual quermassintegrals}. However, it was not until the groundbreaking work \cite{MR3573332} of Huang-Lutwak-Yang-Zhang (Huang-LYZ) that the geometric measures associated with dual quermassintegrals were revealed. This led to \emph{dual curvature measures} dual to Federer's curvature measures. The Minkowski problem for dual curvature measures, now known as the dual Minkowski problem, has been the focus in convex geometry and fully nonlinear elliptic PDEs for the last couple of years and has already led to a number of papers in a short period of time. See, for example, B\"{o}r\"{o}czky-Henk-Pollehn \cite{MR3825606}, Chen-Huang-Zhao \cite{MR3953117}, Chen-Li \cite{MR3818073}, Gardner-Hug-Weil-Xing-Ye \cite{MR3882970}, Henk-Pollehn \cite{MR3725875}, Li-Sheng-Wang \cite{MR4055992}, Zhao \cite{MR3880233}. It is important to note that the list is by no means exhaustive.

Recently, Lutwak-Xi-Yang-Zhang \cite{XLYZ} introduced the first Minkowski-type problem in integral geometry, known as the chord Minkowski problem.

Despite many cases of the $L_p$ Minkowski problem and the $(L_p)$ dual Minkowski problem still being outstanding, one of the properties enjoyed across all these problems is that the geometric measures involved always possess certain homogeneity. Thus, if the prescribed measure $\mu$ is proportional to spherical Lebesgue measure, then there exists one and only one constant solution (a unique centered ball). However, for the Gaussian Minkowski problem, if $f\equiv C>0$, depending on the values of $C$, equation \eqref{eq 3.28.7} can have a unique constant solution, or precisely two constant solutions, or no constant solution at all. This is a result of the function $t^{n-1}e^{-\frac{t^2}{2}}$ being strictly increasing and then strictly decreasing as $t$ increases from $0$ to $\infty$. 

The above observed phenomenon suggests that when $|\mu|$ is sufficiently big, equation \eqref{eq local 00005} has no solution. This sets the Gaussian Minkowski problem apart from the above mentioned Minkowski-type problems. In particular, By the works of Ball \cite{MR1243336} and Nazarov \cite{MR2083397}, the Gaussian surface area of a convex body $K$ in dimension $n$ is asymptotically bounded by $n^\frac{1}{4}$. Therefore, a completely new type of condition is needed if the Gaussian Minkowski problem is to be solved fully. Additionally, examples in the Appendix of \cite{MR4252759} suggests that such a condition is nonlinear, meaning that if both $\mu$ and $\nu$ allow for a solution in \eqref{eq local 00005}, it is not necessarily true that $t\mu+(1-t)\nu$ allows for a solution in \eqref{eq local 00005} for a given $t\in (0,1)$. This was not the case with the $L_p$ Minkowski problem or the $(L_p)$ dual Minkowski problem. 

In the current work, we study the planar Gaussian Minkowski problem. Our first result is regarding solutions to the case where $\mu$ is proportional to spherical Lebesgue measure. In particular, we study nonnegative solutions $h$ to

\begin{equation}
\label{eq 3.22.1}
	 \frac{1}{2\pi}
 e^{-\frac{h'^2+h^2}{2}}(h''+h)=C,
 \end{equation}
on $\s$.

\begin{theorem}
\label{main uniqueness theorem}
	Let $K$ be a convex body in $\mathbb{R}^2$ that contains the origin. If the Gaussian surface area measure of $K$ is proportional to the spherical Lebesgue measure; that is, there exists $C>0$ such that $h=h_K$ is a nonnegative solution to \eqref{eq 3.22.1}, then $K$ has to be a centered disk, or, equivalently, $h$ has to be a constant solution. In particular,
	\begin{enumerate}
		\item If $0<C<e^{-\frac{1}{2}}/(2\pi)$, then there are precisely two solutions;
		\item If $C=e^{-\frac{1}{2}}/(2\pi)$, then there is a unique solution;
		\item If $C>e^{-\frac{1}{2}}/(2\pi)$, then there are no solutions.
	\end{enumerate}
\end{theorem}

The idea of the proof is to convert the existence of nonconstant solutions to a problem involving the estimates of a carefully chosen integral. This is motivated by Andrews \cite{MR1949167}. Recently, the approach of Andrews was adapted to consider the number of solutions to the planar ($L_p$) dual Minkowski problem, see Liu-Lu \cite{liu2022number} and Li-Wan \cite{li2022classification}.

Using Theorem \ref{main uniqueness theorem}, a degree-theoretic approach is used to establish the existence of solutions to the following planar Gaussian Minkowski problem in the origin symmetric setting. 

\begin{theorem}[Existence of smooth, small solutions]
\label{thm existence1}
Let $0<\alpha<1$ and $f\in C^{2,\alpha}(\s)$ be a positive even function with $\|f\|_{L^1}<\frac{1}{\sqrt{2\pi}}$. Then, there exists a $C^{4,\alpha}$, origin-symmetric $K$ with $\gamma_2(K)<\frac{1}{2}$ such that its support function $h$ solves
\begin{equation}
	\frac{1}{{2\pi}} e^{-\frac{h'^2+h^2}{2}}(h''+h)=f.
\end{equation}
\end{theorem}
Recall that isoperimetric inequalities in $(\mathbb{R}^2, \gamma_2)$ is not particularly helpful outside the class $\gamma_2(K)\geq 1/2$. This conceptually makes it much harder to obtain $C^0$ estimates. This is precisely the reason that we need to assume that $f$ is bounded both from above and from below by a positive constant. In the setting of Theorem \ref{thm existence1}, this is guaranteed by the fact that $f\in C^{2,\alpha}(\s)$ is a positive function. 

Using an approximation argument, in combination with \cite[Theorem 1.4]{MR4252759}, we immediately have
\begin{theorem}
	Let $f\in L^1(\s)$ be an even function such that $\|f\|_{L^1} <\frac{1}{\sqrt{2\pi}}$. If there exists $\tau>0$ such that $\frac{1}{\tau}<f<\tau$ almost everywhere on $\s$, then there exist at least two origin-symmetric convex bodies $K_1$ and $K_2$ such that
	\begin{equation}
		dS_{\gamma_2, K_1}=dS_{\gamma_2, K_2}(v) = f(v)dv.
	\end{equation}
\end{theorem}

Finally, we remark that Minkowski problems for non-homogeneous measures are often referred to Orlicz-Minkowski-type problems, which has their origin in the work \cite{MR2652213} by Haberl-Lutwak-Yang-Zhang for the Orlicz Minkowski problem that generalizes both the classical Minkowski problem and the $L_p$ Minkowski problem. See also \cite{MR3209355,MR3263511,MR3897433, MR3882970, MR4040624} for additional results and contributions in the Orlicz extension of the classical Brunn-Minkowski theory. When the ambient space is equipped with log-concave measures, this was considered recently in Kryvonos-Langharst \cite{kryvonos2022measure}. However, in all these works, the solution is obtained up to a constant. In the case of Gaussian probability space, for example, for a given finite Borel measure $\mu$ on $\sn$, one tries to find a pair of $(c,K)$ where $c>0$ and $K$ is a convex body such that
\begin{equation}
\label{eq 3.28.8}
	\mu = cS_{\gamma_n, K}.
\end{equation}
Note that with the introduction of this $c$, none of the challenges we mentioned regarding the Gaussian Minkowski problem would appear. But, in general, one would also lose the potential to prove \emph{any} uniqueness result. In particular, Theorem 1.4 in \cite{MR4252759} (or Theorem \ref{thm existence1} in the current work in dimension 2 and under certain restrictions on the regularity of $\mu$) immediately implies that there are \emph{infinitely many} pairs of $(c,K)$ that solve \eqref{eq 3.28.8}. In fact, the number of solutions is as many as the cardinality of $\mathbb{R}$!

\section{Preliminaries}
Some basics, as well as notations, regarding convex bodies will be provided in this section. For a general reference on the theory of convex bodies, the readers are referred to the book \cite{MR3155183} by Schneider.

For a Borel measure $\mu$, we use the standard notation $|\mu|$ to denote its total measure. We will use $\gamma_n$ for the Gaussian probability measure in $\rn$; that is,
\begin{equation}
	\gamma_n(E)=\frac{1}{(2\pi)^\frac{n}{2}}\int_E e^{-\frac{|x|^2}{2}}dx.
\end{equation}

By a convex body in $\rn$, we mean a compact convex subset with nonempty interior. Note that if the convex body is also origin-symmetric, then it necessarily contains the origin as an interior point. 

Let $K$ be a compact convex subset in $\rn$. The support function $h_K$ of $K$ is defined by
\begin{equation}
\label{eq 3.28.5}
	h_K(y)=\max \{x\cdot y:x\in K\}, 
\end{equation}
for each $y\in \rn$. It is straightforward to show that $h_K$ is homogeneous of degree 1 and is sublinear. On the other side, if $f$ is a positive continuous function on $\sn$, the \emph{Wulff shape} $\bm{[}f\bm{]}$ of $f$ is the convex body defined by
\begin{equation}
	\bm{[}f\bm{]}= \{x\in \rn: x\cdot v \leq f(v), \text{ for all }v \in \sn\}.
\end{equation}
It is not hard to see that on $\sn$, we have $h_{[f]}\leq f$. If $f=h_K$, then $[f]=K$. If $K$ is origin-symmetric, by the definition of support function, we have the following useful estimate:
\begin{equation}
	h_K(v)\geq |x\cdot v|, \forall v\in \sn, \forall x\in K.
\end{equation}

The space of convex bodies in $\rn$ can be made into a metric space by considering the Hausdorff metric. Suppose $K_i$ is a sequence of convex bodies in $\rn$. We say $K_i$ converges to a compact convex subset $K\subset \rn$ in Hausdorff metric if
\begin{equation}
\label{eq convergence convex bodies}
\max\{|h_{K_i}(v)-h_K(v)|:v\in S^{n-1}\}\rightarrow 0,
\end{equation}
as $i\rightarrow \infty$.

For a compact convex subset $K$ in $\rn$ and $v \in \sn$, the supporting hyperplane $H(K,v)$ of $K$ at $v$ is given by
\begin{equation*}
H(K,v)=\{x\in K: x\cdot v = h_K(v)\}.
\end{equation*}
By its definition, the supporting hyperplane $H(K,v)$ is non-empty and contains only boundary points of $K$. For $x\in H(K,v)$, we say $v$ is an outer unit normal of $K$ at $x\in \partial K$. 

Since $K$ is convex, for $\mathcal{H}^{n-1}$ almost all $x\in \partial K$, the outer unit normal of $K$ at $x$ is unique. In this case, we use $\nu_K$ to denote the Gauss map that takes $x\in \partial K$ to its unique outer unit normal. Therefore, the map $\nu_K$ is almost everywhere defined on $\partial K$. We use $\nu_K^{-1}$ to denote the inverse Gauss map. Since $K$ is not assumed to be strictly convex, the map $\nu_K^{-1}$ is set-valued map and for each set $\eta\subset \sn$, we have
\begin{equation*}
	\nu_K^{-1}(\eta) = \{x\in \partial K: \text{there exists } v\in \eta \text{ such that } v \text{ is an outer unit normal at }x\}.
\end{equation*} 

Let $K$ be a convex body in $\rn$ that contains the origin as an interior point. The Gaussian surface area measure of $K$, denoted by $S_{\gamma_n,K}$, is a Borel measure on $\sn$ given by
\begin{equation}
\label{eq 3.28.6}
	S_{\gamma_n, K}(\eta) = \frac{1}{(2\pi)^{\frac{n}{2}}}\int_{\nu_K^{-1}(\eta)}e^{-\frac{
	|x|^2}{2}}d\mathcal{H}^{n-1}(x),
\end{equation}
for each Borel measurable $\eta\subset \sn$. It can be shown that $S_{\gamma_n, K}$ is weakly continuous in $K$ (with respect to Hausdorff metric): if $K_i, K$ are convex bodies in $\rn$ that contain the origin as interior point and $K_i$ converges to $K$ in Hausdorff metric, then $S_{\gamma_n, K_i}$ converges weakly to $S_{\gamma_n, K}$. See, for example, Theorem 3.4 in \cite{MR4252759}. 

By a simple calculation, it follows from the definition of Gaussian surface area measure that if $K\in \mathcal{K}_o^n$ is convex, then $S_{\gamma_n, K}$ is absolutely continuous with respect to surface area measure and 
\begin{equation*}
	dS_{\gamma_n, K } = \frac{1}{(\sqrt{2\pi})^n} e^{-\frac{|\nabla h_K|^2+h_K^2}{2}} dS_K.
\end{equation*}
If, in addition, the body $K$ is $C^2$, then $S_{\gamma_n, K}$ is absolutely continuous with respect to spherical Lebesgue measure and 
\begin{equation}
\label{eq local 9093}
	dS_{\gamma_n, K }(v)= \frac{1}{(\sqrt{2\pi})^n} e^{-\frac{|\nabla h_K|^2+h_K^2}{2}} \det (\nabla^2 h_K+h_KI)dv.
\end{equation}

When $P\in \mathcal{K}_o^n$ is a polytope with unit normal vectors $v_i$ with the corresponding faces $F_i$, the Gaussian surface area measure $S_{\gamma_n, P}$ is a discrete measure given by
\begin{equation*}
	S_{\gamma_n, P}(\cdot) = \sum_{i=1}^N \alpha_i \delta_{v_i}(\cdot),  
\end{equation*}
where $\alpha_i$ is given by 
\begin{equation*}
	\alpha_i=\frac{1}{(\sqrt{2\pi})^n}\int_{F_i} e^{-\frac{|x|^2}{2}}d\mathcal{H}^{n-1}(x).
\end{equation*}

Let $K, L$ be two convex bodies in $\rn$. The Erhard inequality states that
	for $0<t<1$, we have
	\begin{equation}
	\label{eq 3.28.4}
		\Psi^{-1}(\gamma_n((1-t)K+tL))\geq (1-t)\Psi^{-1}(\gamma_n(K))+t\Psi^{-1}(\gamma_n(L)).
	\end{equation}
	Here, $(1-t)K+tL=\{(1-t)x+ty:x\in K, y\in L\}$ is the Minkowski combination between $K$ and $L$, and 
	\begin{equation}
	\label{eq local 203}
		\Psi(x) = \frac{1}{\sqrt{2\pi}}\int_{-\infty}^{x} e^{-\frac{t^2}{2}}dt.
	\end{equation}
	Moreover, equality holds if and only if $K=L$.

The Gaussian isoperimetric inequality in $\rn$ states that for every convex body $K$ in $\rn$, we have
\begin{equation}
\label{eq 3.28.1}
	|S_{\gamma_n,K}|\geq \psi (\Psi^{-1}(\gamma_n(K))),
\end{equation}
where $\psi(t)=\frac{1}{\sqrt{2\pi}}e^{-\frac{t^2}{2}}$. Note that the Gaussian isoperimetric inequality follows directly from Ehrhard inequality. In the current work, we require the following special case of \eqref{eq 3.28.1}: if $\gamma_n(K)=\frac{1}{2}$, then $|S_{\gamma_n,K}|\geq \frac{1}{\sqrt{2\pi}}$.

Sections \ref{sec integral} and \ref{sect estimation} are devoted to studying the number of solutions to the equation 
\begin{equation}
	dS_{\gamma_2, K}=Cd\mathcal{H}^{1},
\end{equation}
on the set of convex bodies in $\mathbb{R}^2$ that contain the origin, or equivalently, the number of nonnegative solutions to the equation
\begin{equation}
\label{eq 3.28.2}
	 e^{-\frac{h'^2+h^2}{2}}(h''+h)=c,
\end{equation}

Note that it is quite simple, by studying the monotonicity properties of the function $te^{-\frac{t^2}{2}}$, to see the number of \emph{constant} solutions to \eqref{eq 3.28.2}. For easier reference, we state this as a proposition.

\begin{prop}
\label{prop constant solutions}
	We have
	\begin{enumerate}[(i)]
	\item if $0<c<e^{-\frac{1}{2}}$, there are precisely two constant solutions to \eqref{eq 3.28.2};
	\item if $c=e^{-\frac{1}{2}}$, there is precisely one constant solution to \eqref{eq 3.28.2};
	\item if $c>e^{-\frac{1}{2}}$, there is no constant solution to \eqref{eq 3.28.2}.
\end{enumerate} 
\end{prop}

We remark that if we denote $h_1\equiv r_1, h_2\equiv r_2$ with $r_1>r_2$ to be the two constant solutions to \eqref{eq 3.28.2} when $c\in (0, e^{-\frac{1}{2}})$, then it is simple to see that as $c\rightarrow 0$, we have $r_1\rightarrow \infty$ and $r_2\rightarrow 0$.

\section{An integral associated with number of nonconstant solutions}
\label{sec integral}

In this section, we construct an integral with parameters, the value estimate of which would lead to the number of nonconstant solutions to \eqref{eq 3.22.1}. 

For sake of simplicity, we move the constant $2\pi$ in \eqref{eq 3.22.1} and rewrite the equation as
\begin{equation}
\label{eq 3.22.2}
 e^{-\frac{h'^2+h^2}{2}}(h''+h)=c,
\end{equation}
for some given $c>0$. Based on the regularity theory developed by Caffarelli \cite{MR1005611, MR1038359,MR1038360}, it is immediate that if $h$ is a nonnegative solution to \eqref{eq 3.22.2}, then $h\in C^\infty$. Moreover, 
\begin{equation}
	(e^{-\frac{h'^2+h^2}{2}})'=e^{-\frac{h'^2+h^2}{2}}(h''+h)(-h')\
=-ch',
\end{equation}
where in the second equality, we used the fact that $h$ solves \eqref{eq 3.22.2}. Therefore, there exists a constant $E$ such that
\begin{equation}
\label{eq 3.22.3}
	e^{-\frac{h'^2+h^2}{2}}+ch\equiv E, \qquad \text{on}\ \s.
\end{equation}

Write
\begin{equation}
\label{eq 3.22.6}
		h_0= \min_{\theta\in \s} h(\theta), \quad \text{and} \quad h_1 = \max_{\theta\in \s} h(\theta).
\end{equation}
Note that $h$ being a constant function is equivalent to $h_0=h_1$. Consider the function 
\begin{equation}
	\phi(t) = ct+e^{-\frac{t^2}{2}}.
\end{equation}
By direct computation, 
\begin{equation}
	\phi'(t)=c-te^{-\frac{t^2}{2}}.
\end{equation}

The number of solutions to \eqref{eq 3.22.2} when $c\geq e^{-\frac{1}{2}}$ is immediate based on a simple observation on the monotonicity of $\phi$ in this case.
\begin{theorem}
\label{thm 3.22.1}
	If $c\in [e^{-\frac{1}{2}}, \infty)$ and $h$ is a nonnegative solution to \eqref{eq 3.22.2}, then $h$ has to be a constant solution. In particular, 
	\begin{enumerate}
		\item If $c=e^{-\frac{1}{2}}$, there is exactly one convex body $K$ containing the origin such that $dS_{\gamma_2, K}=cd\mathcal{H}^{n-1}$ and $K$ is a centered disk;
		\item If $c>e^{-\frac{1}{2}}$, there exists no convex body $K$ containing the origin with $dS_{\gamma_2, K}=cd\mathcal{H}^{n-1}$. 
	\end{enumerate}
\end{theorem}
\begin{proof}
	By \eqref{eq 3.22.3}, the definition of $\phi$, and the fact that $h_0$ and $h_1$ are extremal values, we have
	\begin{equation}
		\phi(h_0)=\phi(h_1)=E. 
	\end{equation}
	A simple calculation yields that when $c\geq e^{-\frac{1}{2}}$, the function $\phi$ is strictly monotone on $[0,\infty)$. Therefore $h_0=h_1$, which implies that $h$ is a constant solution. The other claims of this theorem follows from Proposition \ref{prop constant solutions}.
\end{proof}

For simplicity, we will write
\begin{equation}
	g(t)=te^{-\frac{t^2}{2}}.
\end{equation}
Note that $g$ is strictly monotonically increasing on $[0,1]$ and strictly monotonically decreasing on $[1,\infty)$. Moreover $g(1)=\max_{t\in [0,\infty)} g(t)=e^{-\frac{1}{2}}$.

The rest of this section and Section \ref{sect estimation} will focus on showing that there is no nonconstant solutions to \eqref{eq 3.22.2} when $c\in (0,e^{-\frac{1}{2}})$. Note that in this case, the equation 
\begin{equation}
\label{eq 3.22.4}
	g(t)=c
\end{equation}
has exactly two nonnegative solutions, which we shall denote as $m_1$ and $m_2$, with $m_1<1<m_2$.

Towards this end, we assume that $h$ is a nonnegative, nonconstant solution to \eqref{eq 3.22.2}. 

\begin{lemma}
\label{lemma 3.22.1}
	If $c\in (0,e^{-\frac{1}{2}})$ and $h$ is a nonnegative, nonconstant solution to \eqref{eq 3.22.2}, then the critical points of $h$ are isolated.
\end{lemma}
\begin{proof}
	Suppose the critical points of $h$ has an accumulation point; that is, there exists a sequence of distinct $\theta_i \in \s$ converging to $\theta_0\in \s$ such that $h'(\theta_i)=0$. Since $h$ is smooth, we conclude that $h'(\theta_0)=0$. Moreover,
	\begin{equation}
		h''(\theta_0) = \lim_{i\rightarrow \infty} \frac{h'(\theta_i)-h'(\theta_0)}{\theta_i-\theta_0}=0.
	\end{equation} 
	Equation \eqref{eq 3.22.2} now implies that 
	\begin{equation}
		g(h(\theta_0)) = h(\theta_0)e^{-\frac{h^2(\theta_0)}{2}}=c.
	\end{equation}
	Therefore $h(\theta_0)=m_1$ or $h(\theta_0)=m_2$, where $m_1$ and $m_2$ are nonnegative solutions to \eqref{eq 3.22.4}. Recall that $h'(\theta_0)=0$. Thus, $h$ is a solution to the initial value problem
	\begin{equation}
	\label{eq 3.22.5}
		\begin{cases}
			e^{-\frac{h'^2+h^2}{2}}(h''+h)=c,\\
			h(\theta_0)=m,\quad h'(\theta_0)=0,
		\end{cases}
	\end{equation}
	where $m$ is either $m_1$ or $m_2$. Therefore, $h\equiv m$ is the unique solution to \eqref{eq 3.22.5}, which contradicts with the assumption that $h$ is nonconstant.
\end{proof}

\begin{lemma}
	If $c\in (0,e^{-\frac{1}{2}})$ and $h$ is a nonnegative, nonconstant solution to \eqref{eq 3.22.2}, then $h_0<m_1<h_1\leq m_2$, where $m_1$ and $m_2$ are nonnegative solutions to \eqref{eq 3.22.4} and $h_0, h_1$ are as given in \eqref{eq 3.22.6}.
\end{lemma}
\begin{proof}
	Note that $\phi'(t)=c-g(t)$. This implies that $\phi$ is strictly increasing on $[0,m_1]$ and on $[m_2,\infty)$, whereas strictly decreasing on $[m_1,m_2]$. Argued in the same way as in the proof of Theorem \ref{thm 3.22.1}, we conclude that 
	\begin{equation}
	\label{eq 3.23.1}
		\phi(h_0)=\phi(h_1)=E,
	\end{equation} 
	where $h_0<h_1$.
	
	Let $h\in (h_0,h_1)$ be arbitrary. By intermediate value theorem, there exists $\theta\in \s$ such that $h(\theta)=h$. By \eqref{eq 3.22.2} and \eqref{eq 3.22.3}, we have
	\begin{equation}
		\phi(h)=\phi(h(\theta)) = ch(\theta)+e^{-\frac{h^2(\theta)}{2}}\geq ch(\theta)+e^{-\frac{h'^2(\theta)+h^2(\theta)}{2}}=E.
	\end{equation}
	This, when combined with the monotonicity of $\phi$, immediately gives us the desired result.
\end{proof}

In particular, this suggests that 
\begin{equation}
	\phi'(h_0)>0, \quad \text{ and } \quad \phi'(h_1)\leq 0.	
\end{equation}

Motivated by this, we give the following definition.

\begin{defi}
	A pair of constants $a<b$ is called a \emph{good pair} with respect to $c$, if it satisfies $\phi(a)=\phi(b)$ and $\phi'(a)>0$, $\phi'(b)\leq 0$.
\end{defi}
It is immediate that for $t\in (a,b)$, we have $\phi(t)>\phi(a)=\phi(b)$. In particular, if $h$ is a nonnegative, nonconstant solution to \eqref{eq 3.22.2}, then $(h_0, h_1)$ is a good pair.

\begin{lemma}
\label{lemma 3.22.2}
	If $c\in (0,e^{-\frac{1}{2}})$ and $h$ is a nonnegative, nonconstant solution to \eqref{eq 3.22.2}, then the critical points of $h$ are either minimum or maximum points.
\end{lemma}
\begin{proof}
	If $\theta_0$ is a critical point of $h$, then $h'(\theta_0)=0$ and therefore \eqref{eq 3.22.3} implies that $\phi(h(\theta_0))=E$. This, in combination with the fact that $h_0\leq h(\theta_0)\leq h_1$ and that $(h_0,h_1)$ is a good pair, implies that $h(\theta_0)$ is either $h_0$ or $h_1$. In another word, it is either a minimum or a maximum point.
\end{proof}

The following corollary is immediate from Lemmas \ref{lemma 3.22.1} and \ref{lemma 3.22.2}.
\begin{coro}
\label{coro 3.23.1}
	If $c\in (0,e^{-\frac{1}{2}})$ and $h$ is a nonnegative, nonconstant solution to \eqref{eq 3.22.2}, then minimum points and maximum points of $h$ alternate, and there are finitely many of them. In addition, the number of critical points of $h$ is an even number.
\end{coro}

Let $h$ be a nonnegative, nonconstant solution to \eqref{eq 3.22.2}. Assume $\theta_0$ is a minimum point of $h$; \emph{i.e., } $h(\theta_0)=h_0$ and $\theta_1\in [0,2\pi)$ be the nearest maximum point of $h$; \emph{i.e.,} $h(\theta_1)=h_1$. Without loss of generality, we assume $\theta_1>\theta_0$. Note that this implies $h'(\theta)>0$ for $\theta\in (\theta_0,\theta_1)$. By \eqref{eq 3.22.3}, we have
\begin{equation}
	h'(\theta)=\sqrt{-h^2(\theta)-2\log(E-ch(\theta))}.
\end{equation}
Making the change of variable $u=h(\theta)$, we get
\begin{equation}
	\theta_1-\theta_0=\int_{\theta_0}^{\theta_1} d\theta=\int_{h_0}^{h_1}\frac{1}{h'(\theta(u))}du= \int_{h_0}^{h_1}\frac{1}{\sqrt{-u^2-2\log(E-cu)}}du. 
\end{equation}
Note that this suggests that $\theta_1-\theta_0$ only depends on the values of $h_0$ and $h_1$ (in another word, independent of the specific location of the minimum/maximum point). Let $r=h_1-h_0$ and make the change of variable $t=\frac{u-h_0}{r}$, we have by \eqref{eq 3.23.1} that 
\begin{equation}
\label{eq 3.23.3}
	\theta_1-\theta_0=\int_{0}^{1}\frac{r}{\sqrt{-(tr+h_0)^2-2\log(e^{-\frac{h^2_0}{2}}-ctr)}}dt :=\Theta(c,h_0, r).
\end{equation}
This, when combined with Corollary \ref{coro 3.23.1}, immediately implies the following crucial lemma.

\begin{lemma}
\label{lemma 3.23.1}
	Let $c\in (0, e^{-\frac{1}{2}})$. If the set 
	\begin{equation}
		\left\{(h_0, r): \Theta(c,h_0, r)=\frac{\pi}{k} \text{ for some positive integer }k \text{ and } (h_0, h_0+r) \text{ is a good pair w.r.t. } c\right\}
	\end{equation}
	is empty, then the only nonnegative solutions to \eqref{eq 3.22.2} are constant solutions.
\end{lemma}
\begin{proof}
	If $h$ is a nonnegative, nonconstant solution to \eqref{eq 3.22.2}, then according to \eqref{eq 3.23.3}, the difference in $\theta$ between two consecutive critical points of $h$ is the same, regardless of the location of critical points. By Corollary \ref{coro 3.23.1}, there exists a positive integer $k$ such that $\Theta(c,h_0, r)=\frac{\pi}{k}$. By definition of good pair, $(h_0,h_0+r)$ is a good pair with respect to $c$. Therefore, the set is nonempty, yielding a contradiction. 
\end{proof}

\begin{remark}
	It can actually be shown that the cardinality of the set in Lemma \ref{lemma 3.23.1} is precisely the number of nonnegative, nonconstant solutions (up to rotation) to \eqref{eq 3.22.2}. But, this is not necessary in the current work.
\end{remark}

\section{Estimating $\Theta$}
\label{sect estimation}

The main purpose of this section is to provide estimations regarding the integral $\Theta(c, h_0, r)$ subject to the constraint that $(h_0, h_0+r)$ is a good pair with respect to $c$. 

In the rest of this section, we will constantly fix one of the parameters $c$, $h_0$, or $r$. The constraint that $(h_0, h_0+r)$ is a good pair with respect to $c$ now suggests that one of the remaining two parameters will uniquely determine the other. It is important to note that when one of the parameters is fixed, one might not be able to arbitrarily pick the other parameters. For example, when $c\rightarrow 0^+$, $h_0$ also has to approach $0$. Details such as this will be provided when needed.

 For now, we fix $c\in (0, e^{-\frac{1}{2}})$ and view $h_0$ as a function of $r$. To see why this is possible, recall that $0<m_1<1<m_2$ are the two nonnegative solutions of 
 \begin{equation}
 	te^{-\frac{t^2}{2}}=c. 
 \end{equation}
 If $\phi(m_2)>\phi(0)=1$, then by monotonicity properties of $\phi$, there exists a unique $0<q<m_1$ such that $\phi(q)=\phi(m_2)$. Moreover, for every $h_0\in [q, m_1)$, there exists a unique $r>0$ such that $h_0+r\in (m_1,m_2]$ and $\phi(h_0+r)=\phi(h_0)$. By definition, such a pair $(h_0, h_0+r)$ is a good pair with respect to $c$. It is also simple to see that if $h_0\notin [q, m_1)$, then there is no $r>0$ such that $(h_0, h_0+r)$ is a good pair.
 
  On the other hand, if $\phi(m_2)\leq \phi(0)=1$, using a similar argument, we conclude that there exists a unique $r>0$ such that $(h_0, h_0+r)$ is a good pair with respect to $c$ if and only if $h_0\in [0,m_1)$. 
  
  Note that by the monotonicity properties of $\phi$, $r$ decreases strictly as $h_0$ increases. Set 
  \begin{equation}
  	\mathbb{H}=\begin{cases}
  		[0,m_1),  \ \  {\phi(m_2)\leq \phi(0)},\\
[q,m_1),  \ \  {\phi(m_2)> \phi(0)}.\\
  	\end{cases}
  \end{equation}
  For $h_0\in \mathbb{H}$, we write $r$ as a function in $h_0$; that is $r=r(h_0)$. It is simple to see that $r(h_0)$ is continuous in $h_0$ and $r\rightarrow 0$ as $h_0\rightarrow m_1$. Now set
  \begin{equation}
  	r_c=
\begin{cases}
r(0),  \ \  {\phi(m_2)\leq \phi(0)},\\
r(q),  \ \  {\phi(m_2)> \phi(0)}.\\
\end{cases}
  \end{equation}
  Thus $r(h_0):\mathbb{H}\rightarrow (0,r_C]$ is a bijection and consequently we may write $h_0=h_0(r)$ for $r\in (0, r_c]$. Note that since $\phi$ is smooth and $h_0, r$ are implicitly defined by $\phi(h_0+r)=\phi(h_0)$, the function $h_0(r)$ is also smooth.
  
  We shall use the fact that $\phi(h_0+r)=\phi(h_0)$ is equivalent to 
  \begin{equation}
  \label{eq 3.23.5}
  	e^{-\frac{(h_0+r)^2}{2}}-e^{-\frac{h_0^2}{2}}+cr=0.
  \end{equation}
  
\begin{lemma}
\label{lemma 3.23.2}
 Let $c\in(0,e^{-\frac{1}{2}})$ and $h_0(r)$ be such that $(h_0(r), h_0(r)+r)$ is a good pair with respect to $c$ for $r\in (0, r_c]$. Then
 \begin{equation}
 	\lim_{r\rightarrow0}h'_0(r)=-\frac{1}{2},
 \end{equation}
 and 
 \begin{equation}
 	\lim_{r\to0}rh''_0(r)=0.
 \end{equation}
\end{lemma}

\begin{proof}
Recall that $h_0(r)\rightarrow m_1$ as $r\rightarrow 0$. We therefore, define $h_0(0)=m_1$ so that $h_0(r)$ is continuous on $[0,r_c]$. Since $c\in (0, e^{-\frac{1}{2}})$, we have $m_1<1$ and consequently $\phi''(m_1)=(m_1^2-1)e^{-\frac{m_1^2}{2}}\neq 0$. Taylor's theorem when applied to $\phi$ at $m_1$ gives
\begin{equation}
	\phi(t)=\phi(m_1)+\frac{1}{2}\phi''(m_1)(t-m_1)^2+o((t-m_1)^2).
\end{equation}
Here we used the fact that $\phi'(m_1)=0$. 
Hence
\begin{equation}
	\phi(h_0(r))=\phi(m_1)+\frac{1}{2}\phi''(m_1)(h_0(r)-m_1)^2+o((h_0(r)-m_1)^2)
\end{equation}
and
\begin{equation}
	\phi(h_0(r)+r)=\phi(m_1)+\frac{1}{2}\phi''(m_1)(h_0(r)+r-m_1)^2+o((h_0(r)+r-m_1)^2).
\end{equation} 
By definition of good pair, we have $\phi(h_0(r))=\phi(h_0(r)+r)$ and $h_0(r)<m_1<h_0(r)+r$. Consequently, $|h_0(r)-m_1|, |h_0(r)+r-m_1|\leq r$, and
\begin{equation}
	\frac{1}{2}\phi''(m_1)(h_0(r)-m_1)^2=\frac{1}{2}\phi''(m_1)(h_0(r)+r-m_1)^2+o(r^2),
\end{equation} 
This implies
\begin{equation}
	(h_0(r)-m_1)^2=(h_0(r)+r-m_1)^2+o(r^2),
\end{equation}
since $\phi''(m_1)\neq 0$. Therefore, 
\begin{equation}
h_0(r)-h_0(0)=h_0(r)-m_1=-\frac{1}{2}r+o(r).
\end{equation}
Thus
\begin{equation}
\label{eq 3.23.6}
	h_0'(0)=-\frac{1}{2}.
\end{equation}

Differentiating \eqref{eq 3.23.5} in $r$, we have
\begin{eqnarray}\label{3.1}
e^{-\frac{(h_0+r)^2}{2}}(h_0+r)-c+[e^{-\frac{(h_0+r)^2}{2}}(h_0+r)-e^{-\frac{h^2_0}{2}}h_0]h'_0=0.
\end{eqnarray}
Note that the definition of $m_1$ and that $h_0(0)=m_1$ implies
\begin{equation}
\label{eq 3.23.7}
	c=g(h_0(0)) \quad \text{ and } g'(h_0(0))>0. 
\end{equation}
where $g$ is given in \eqref{eq 3.22.4}. Hence,
\begin{equation}
\label{eq 3.23.8}
	h_0'(r)=\frac{g(h_0(0))-g(h_0(r))}{g(h_0(r)+r)-g(h_0(r))}-1.
\end{equation}
By mean value theorem, there exists $s_1$ between $h_0(0)$ and $h_0(r)$, and $s_2$ between $h_0(r)$ and $h_0(r)+r$ such that
\begin{equation}
	h_0'(r) = \frac{g'(s_1)(h_0(0)-h_0(r))}{g'(s_2)r}-1.
\end{equation}
Let $r\rightarrow 0$. By \eqref{eq 3.23.6} and \eqref{eq 3.23.7}, we conclude
\begin{equation}
\label{eq 3.23.9}
	\lim_{r\rightarrow 0}h_0'(r)=-\frac{1}{2}.
\end{equation}

For the second equality in the statement of this lemma, a straightforward computation shows that
\begin{eqnarray*}
h_0''(r)=
\frac{-g'(h_0(r))h_0'(r)}{g(h_0(r)+r)-g(h_0(r))}-\frac{\left(g(h_0(0))-g(h_0(r))\right)\bigg(g'(h_0(r)+r)(h_0'(r)+1)-g'(h_0(r))h_0'(r)\bigg)}{(g(h_0(r)+r)-g(h_0(r)))^2}.\\
\end{eqnarray*}
By \eqref{eq 3.23.8} and mean value theorem, we have 
$$rh_0''(r)=\frac{-g'(h_0(r))h_0'(r)}{g'(s)}-(h_0'(r)+1)
\frac{g'(h_0(r)+r)(h_0'(r)+1)-g'(h_0(r))h_0'(r)}{g'(s)},
$$
for $s\in (h_0(r), h_0(r)+r)$.
Let $r\rightarrow 0$. By \eqref{eq 3.23.9}, we have
$${\lim_{r\to0}}rh''_0(r)=0.$$
\end{proof}

Using Lemma \ref{lemma 3.23.2} we can prove the following estimate.
\begin{lemma}\label{inflim}
Let $c\in(0,e^{-\frac{1}{2}})$ and $h_0(r)$ be such that $(h_0(r), h_0(r)+r)$ is a good pair with respect to $c$ for $r\in (0, r_c]$. Denote
\begin{equation}
	\Theta(r) = \Theta(c, h_0(r),r).
\end{equation}
Then,
\begin{equation}
	\liminf_{r\to0}\,\Theta(r)\geq\frac{\pi}{\sqrt{1-m_1^2}}
\end{equation}
\end{lemma}                                   
\begin{proof} 
As in Lemma \ref{lemma 3.23.2}, we set $h_0(0)=m_1$. 

By Fatou Lemma,
\begin{equation}\label{1}
\liminf_{r\to0}\Theta(r)\,\geq\int_{0}^{1}\liminf_{r\to0}\sqrt{\frac{r^2}{-(tr+h_0)^2-2\log(e^{-\frac{h^2_0}{2}}-ctr)}}dt
\end{equation}

We claim that 
\begin{equation}
\label{eq 3.23.10}
	{\lim_{r\to0}}\frac{r^2}{-2\log(e^{-\frac{h^{2}_0(r)}{2}}-ctr)-(tr+h_0)^2}=\frac{-1}{(1-h^{2}_0(0))(t^2-t)}.
\end{equation}
Therefore, 
\begin{equation}
\begin{aligned}
	&\int_{0}^{1}\liminf_{r\to0}\sqrt{\frac{r^2}{-(tr+h_0)^2-2\log(e^{-\frac{h^2_0}{2}}-ctr)}}dt\\
	=&\int_{0}^{1}\sqrt{\frac{-1}{(1-h^{2}_0(0))(t^2-t)}}dt\\
=&\frac{1}{\sqrt{(1-h^{2}_0(0))}}\int_{0}^{1}\frac{1}{\sqrt{(-t^2+t)}}dt\\
=&\frac{\pi}{\sqrt{1-{h^2_0}(0)}}.
\end{aligned}	
\end{equation}
The desired result follows immediately.

It remains to show \eqref{eq 3.23.10}, which follows from repeatedly applying L'H\^{o}pital's rule. For simplicity, Denote  $d(r,t):=e^{-\frac{h_0^2(r)}{2}}-crt$. By L'H\^{o}pital's rule, 
\begin{equation*}
\begin{aligned}
&{\lim_{r\to0^{+}}}\frac{r^2}{-2\log(d(r,t))-(tr+h_0)^2}\\
&={\lim_{r\to0^{+}}}\frac{r}{-\frac{e^{-\frac{h^{2}_0(r)}{2}}(-h_0(r))h'_0(r)-ct}{d(r,t)}-(tr+h_0(r))(t+h'_0(r))}\\
&=e^{-\frac{h^{2}_0(0)}{2}}{\lim_{r\to0^{+}}}\frac{-r}{e^{-\frac{h^{2}_0(r)}{2}}(-h_0(r))h'_0(r)-ct+(tr+h_0(r))(t+h'_0(r))d(r,t)}\\
\end{aligned}
\end{equation*}
Since $e^{-\frac{h^{2}_0(0)}{2}}h_0(0)=c$, by Lemma \ref{lemma 3.23.2}, we have
\begin{equation*}
\begin{aligned}
&{\lim_{r\to0}}\left(e^{-\frac{h^{2}_0(r)}{2}}(-h_0(r))h'_0(r)-ct+(tr+h_0(r))(t+h'_0(r))d(r,t)\right)\\
&=\frac{c}{2}-ct-c(t-\frac{1}{2})\\
&=0.
\end{aligned}
\end{equation*}
Thus, we may use L'H\^{o}pital's rule again. By direct computation, 
\begin{equation*}
\begin{aligned}
&{\lim_{r\to0}}\frac{d}{dr}\left(e^{-\frac{h^{2}_0(r)}{2}}(-h_0(r))h'_0(r)-ct+(tr+h_0(r))(t+h'_0(r))d(r,t)\right)\\
=&{\lim_{r\to0}}\frac{d}{dr}\left(e^{-\frac{h^{2}_0(r)}{2}}(t^2r+th_0(r)+trh'_0(r))-ct-(tr+h_0(r))(t+h'_0(r))ctr\right)\\
=&e^{-\frac{h^{2}_0(0)}{2}}(t^2+2t{\lim_{r\to0}}h'_0(r))-cth_0(0)(2{\lim_{r\to0}}h'_0(r)+t)\\
=&\left(e^{-\frac{h^{2}_0(0)}{2}}-ch_0(0)\right)(t^2-t),\\
\end{aligned}
\end{equation*}
where in the the second and the third equalities, we used Lemma \ref{lemma 3.23.2} and the fact that $e^{-\frac{h^{2}_0(0)}{2}}h_0(0)=c$.

By L'H\^opital's rule,
\begin{equation*}
\begin{aligned}
&{\lim_{r\to0}}\frac{r^2}{-2\log(e^{-\frac{h^{2}_0(r)}{2}}-ctr)-(tr+h_0)^2}\\
=&e^{-\frac{h^{2}_0(0)}{2}}\frac{-1}{\left(e^{-\frac{h^{2}_0(0)}{2}}-ch_0(0)\right)(t^2-t)}\\
=&\frac{-1}{t^2-t}\frac{1}{1-h^2_0(0)},\\
\end{aligned}
\end{equation*}
where the last equality is due to $e^{-\frac{h^{2}_0(0)}{2}}h_0(0)=c$.
\end{proof} 

We now fix $h_0\in [0,1)$ and use the fact that $(h_0, h_0+r)$ is a good pair with respect to $c$ to represent $c$ as a function in $r$. Note that $c$ is determined by $r$ via \eqref{eq 3.23.5}. However, it is not always true that for any $r>0$ , the $c$ determined by  \eqref{eq 3.23.5} will make it true that $(h_0, h_0+r)$ is a good pair with respect to $c$. We need several lemmas to determine the allowable values of $r$.

\begin{lemma}
\label{lemma 3.23.3}
	Let $h_0\in [0,1)$. If $r_*>0$ and $c_*\in (0, e^{-\frac{1}{2}})$ is such that $(h_0, h_0+r_*)$ is a good pair with respect to $c_*$, then for every $0<r<r_*$, there exists $c\in (0, e^{-\frac{1}{2}})$ such that $(h_0 ,h_0+r)$ is a good pair with respect to $c$. 
\end{lemma}
\begin{proof}
	Since $c$ here is changing, we write $\phi$ as $\phi_c$ to emphasize its dependence on $c$. Similarly, we also write $m_1(c)$ and $m_2(c)$. 
	
	Note that for each $c\in (h_0e^{-\frac{h_0^2}{2}}, c_*)$, by definition of $\phi_c$ and $m_1(c)$, we have $h_0<m_1(c)$. Here, the fact that $h_0e^{-\frac{h_0^2}{2}}<c_*$ follows from that $(h_0, h_0+r_*)$ is a good pair with respect to $c_*$. We claim that $\phi_c(h_0)\geq \phi_c(m_2(c))$. Indeed, 
	\begin{equation}
		\begin{aligned}
			\phi_c(h_0)&=ch_0+e^{-\frac{h_0^2}{2}}\\
			&=\phi_{c_*}(h_0)+(c-c_*)h_0\\
			&\geq \phi_{c_*}(m_2(c_*))+(c-c_*)h_0\\
			&=\phi_{c}(m_2(c_*))+[\phi_{c_*}(m_2(c_*))-\phi_{c}(m_2(c_*))+(c-c_*)h_0]\\
			&=\phi_{c}(m_2(c_*))+(c-c_*)(h_0-m_2(c_*))\\
			&>\phi_c(m_2(c_*))\\
			&\geq \phi_c(m_2(c)).
		\end{aligned}
	\end{equation}  
	Here, in the second to last inequality, we used the fact that $h_0<1<m_2(c_*)$. In the last inequality, we used the fact that $m_2(c)>m_2(c_*)$ and that $\phi_c(t)$ is strictly decreasing for $t\in [m_2(c_*), m_2(c)]$.
	
	Hence, by monotonicity properties of $\phi_c$, for every $c\in (h_0e^{-\frac{h_0^2}{2}}, c_*)$, there exists a unique $r>0$ such that the pair $(h_0, h_0+r)$ is a good pair with respect to $c$. Note that when $c\rightarrow h_0e^{-\frac{h_0^2}{2}}$, we have $r\rightarrow 0$. Also note that $r$ depends continuously on $c$. The desired result now follows from intermediate value theorem. 
\end{proof}

\begin{lemma}
\label{lemma 3.23.4}
	Let $h_0\in [0,1)$. Denote 
	\begin{equation}
		r_{h_0}=\sup\{r>0: \text{there exists } c\in (0, e^{-\frac{1}{2}}) \text{ such that } (h_0, h_0+r) \text{ is a good pair with respect to }c \}.
	\end{equation}
	Then
	\begin{equation}
		0<r_{h_0}<\infty.
	\end{equation}
\end{lemma}
\begin{proof}
	We first claim that the set is nonempty. Indeed, for $h_0\in (0,1)$, choose $r_0>0$ such that $h_0+r_0<1$. Let 
	\begin{equation}
		c_0=\frac{-e^{-\frac{(h_0+r_0)^2}{2}}+e^{-\frac{h_0^2}{2}}}{r_0}.
	\end{equation}
	Note that when $r\rightarrow 0$, we have $c_0\rightarrow h_0e^{-\frac{h_0^2}{2}}<e^{-\frac{1}{2}}$. Therefore, by choosing sufficiently small $r_0>0$, we have $c_0\in (0, e^{-\frac{1}{2}})$. By the choice of $c_0$, we have $\phi_{c_0}(h_0)=\phi_{c_0}(h_0+r_0)$. Note that $h_0+r_0<1<m_2(c_0)$. Now the monotonicity properties of $\phi_{c_0}$ guarantee that $(h_0, h_0+r_0)$ is a good pair with respect to $c$. Hence, $r_{h_0}>0$. 
	
	To see why $r_{h_0}<\infty$, note that the fact that $(h_0, h_0+r)$ is a good pair with respect to $c$ implies
	\begin{equation}
		\phi_c'(h_0)>0,\quad \text{and} \quad \phi_c'(h_0+r)\leq 0.
	\end{equation}  
	Consequently, we have
	\begin{equation}
		(h_0+r)e^{-\frac{(h_0+r)^2}{2}}>h_0e^{-\frac{h_0^2}{2}.}
	\end{equation}
	This implies that $r$ is bounded from above and therefore $r_{h_0}<\infty$.
\end{proof}

Lemmas \ref{lemma 3.23.3} and \ref{lemma 3.23.4} imply that for each fixed $h_0\in [0,1)$, if $r\in (0,r_{h_0})$, the pair $(h_0, h_0+r)$ is a good pair with respect to $c=c(r)$ where
\begin{equation}
\label{eq 3.24.2}
	c(r)=\frac{-e^{-\frac{(h_0+r)^2}{2}}+e^{-\frac{h_0^2}{2}}}{r}.
\end{equation}
It is simple to compute
\begin{equation}\label{eq 3.24.1}
	c'(r)  = \frac{(h_0+r)e^{-\frac{(h_0+r)^2}{2}}-c(r)}{r}
                  =-\frac{\phi'(h_0+r)}{r}
                  \geq 0,
\end{equation}
	where the last inequality is due to the fact that $(h_0, h_0+r)$ is a good pair with respect to $c$.


\begin{lemma}
\label{lemma increasing in r}
	Let $h_0\in [0,1)$. The integral
	\begin{equation}
		\Theta(r)=\Theta(c(r), h_0, r)
	\end{equation}
	is increasing in $(0, r_{h_0})$.
\end{lemma}
\begin{proof}
	Denote 
	\begin{equation}
		\eta(r,t)=-2\log(e^{-\frac{1}{2}h_0^2}-c(r)rt)-(rt+h_0)^2.
	\end{equation}
	By the definition of $\Theta(r)$ (see \eqref{eq 3.23.3}), it suffices to show that $\frac{r}{\eta^{\frac{1}{2}}(r,t)}$ is increasing in $r$ for every fixed $t\in (0,1)$. Since
	\begin{equation}
	\label{eq 3.24.3}
		\frac{\partial}{\partial r}\left(\frac{r}{\eta^{\frac{1}{2}}(r,t)}\right)=\frac{1}{\eta^{\frac{3}{2}}(r,t)}[\eta(r,t)-\frac{1}{2}r\eta'(r,t)],
	\end{equation} 
	where the derivative here and in the rest of the proof is always taken with respect to $r$.
	
	Denote
	\begin{equation}
		\omega(r,t)=e^{-\frac{h_0^2}{2}}-c(r)rt.
	\end{equation}
	Note that $\omega(r,t)>0$ for $t\in [0,1]$. By \eqref{eq 3.24.1} and \eqref{eq 3.24.2},
	\begin{equation*}
		\begin{aligned}
			\eta'(r,t)
			&=2\frac{c'(r)rt+c(r)t}{\omega(r,t)}-2(rt+h_0)t\\
			&=2\frac{[(h_0+r)e^{-\frac{(h_0+r)^2}{2}}-c(r)]t+c(r)t}{\omega(r,t)}-2(rt+h_0)t\\
			&=2\frac{[(h_0+r)(e^{-\frac{h_0^2}{2}}-c(r)r)-c(r)]t+c(r)t}{\omega(r,t)}-2(rt+h_0)t,\\
		\end{aligned}
	\end{equation*}
	and consequently,
	\begin{equation}
	\label{eq 3.24.4}
		\begin{aligned}
			&\eta(r,t)-\frac{1}{2}r\eta'(r,t)\\=&-2\log \omega(r,t)-(rt+h_0)^2\\
&-\frac{[(h_0+r)(e^{-\frac{h_0^2}{2}}-c(r)r)-c(r)]tr+c(r)tr}{\omega(r,t)}+(rt+h_0)tr.
		\end{aligned}
	\end{equation}
	Denote $G(r,t)=[\eta(r,t)-\frac{1}{2}r\eta'(r,t)]\omega(r,t)$. Since $\omega(r,t)>0$, by \eqref{eq 3.24.3} and \eqref{eq 3.24.4}, to prove the desired result, it suffices to show $G(r,t)\geq 0$. Note that 
	\begin{equation*}
\begin{aligned}
G(r,t)&=-2\omega(r,t)\log(\omega(r,t))-\omega(r,t)(rt+h_0)^2\\
&-[(h_0+r)(e^{-\frac{h_0^2}{2}}-c(r)r)-c(r)]tr-c(r)tr+(rt+h_0)tr\omega(r,t)\\
&=-2\omega(r,t)\log(\omega(r,t))+t^2r^2c(r)h_0\\
&+t[c(r)rh^2_0+c(r)r^2h_0+r^3c(r)-2rh_0e^{-\frac{h^2_0}{2}}-r^2e^{-\frac{h^2_0}{2}}]-h^2_0e^{-\frac{h^2_0}{2}}.
\end{aligned}
\end{equation*} 
Thus,
\begin{equation}
	\begin{aligned}
 \partial_t G(r,t)=&2c(r)r\log(\omega(r,t))+2c(r)r+2tr^2c(r)h_0+c(r)rh^2_0\\
&+c(r)r^2h_0+r^3c(r)-2rh_0e^{-\frac{h^2_0}{2}}-r^2e^{-\frac{h^2_0}{2}}
\end{aligned}
\end{equation}
and
\begin{equation}
\begin{aligned}
\frac{{\partial^2 G(r,t)}}{\partial t^2}&=\frac{-2{c(r)^2}{r^2}}{\omega(r,t)}+2r^2c(r)h_0\\
&=\frac{1}{\omega(r,t)}[2r^2c(r)h_0(e^{-\frac{h^2_0}{2}}-c(r)rt)-2c(r)^2r^2]\\
&=\frac{1}{\omega(r,t)}[2r^2c(r)(h_0e^{-\frac{h^2_0}{2}}-c(r))-2r^3c(r)^2h_0t]\\
&\leq0,
\end{aligned}
\end{equation}
where in the last inequality, we used the fact that $\phi'(h_0)=c-h_0e^{-\frac{h_0^2}{2}}>0$ due to that $(h_0, h_0+r)$ is a good pair. Hence, $G(r,t)$ is a concave function in $t$. 

Now, it is straight-forward to compute
\begin{equation}
	\begin{aligned}
		G(r,0)=-2\omega(r,0)\log(\omega(r,0))-h_0^2e^{-\frac{h_0^2}{2}}=0,
	\end{aligned}
\end{equation}
and by \eqref{eq 3.24.2},
\begin{equation}
	\begin{aligned}
		G(r,1)=&e^{-\frac{(h_0+r)^2}{2}}(h_0+r)^2+c(r)r(r+h_0)^2-(h_0+r)^2e^{-\frac{h_0^2}{2}}\\
		 =& 	(r+h_0)^2 (e^{-\frac{(h_0+r)^2}{2}}+c(r)r-e^{-\frac{h_0^2}{2}})\\
		 =&0.
	\end{aligned}
\end{equation}
Since $G$ is concave in $t$, this implies $G(r,t)\geq 0$ for all $t\in [0,1]$. 
\end{proof}

We now fix $r>0$ and consider the dependence between $c$ and $h_0$. Recall that $m_1(c)$ and $m_2(c)$ are two nonnegative solutions to \eqref{eq 3.22.4} and that $m_1(c)<1<m_2(c)$, when $c\in (0,e^{-\frac{1}{2}})$.

Suppose $h_*, r\in (0,1)$ and $c_*\in (0,e^{-\frac{1}{2}})$ are such that $h_*+r<1$ and $(h_*, h_*+r)$ is a good pair with respect to $c_*$. For every $h_0\in (0, h_*]$, set
\begin{equation}
\label{eq 3.24.6}
	c(h_0)=\frac{e^{-\frac{h_0^2}{2}}-e^{-\frac{(h_0+r)^2}{2}}}{r}. 
\end{equation}
Note that by definition of $c=c(h_0)$, we have $\phi_c(h_0)=\phi_c(h_0+r)$. This implies that $c\in (0,e^{-\frac{1}{2}})$, as otherwise, the function $\phi_c(t)$ is strictly increasing, which contradicts with $\phi_c(h_0)=\phi_c(h_0+r)$. Therefore, we have $h_0+r<h_*+r<1<m_2(c)$. By monotonicity properties of $\phi_c$, the pair $(h_0, h_0+r)$ is a good pair with respect to $c(h_0)$. Consequently,
\begin{equation}
\label{eq 3.24.7}
	c'(h_0)=\frac{-h_0e^{-\frac{h_0^2}{2}}+(h_0+r)e^{-\frac{(h_0+r)^2}{2}}}{r}=\frac{\phi_c'(h_0)-\phi_c'(h_0+r)}{r}>0,
\end{equation}
where the last inequality follows from the definition of good pair. Since
\begin{equation}
\label{eq 3.24.9}
	\lim_{h_0\rightarrow 0}c(h_0)=\frac{1-e^{\frac{r^2}{2}}}{r}:=c_r,
\end{equation}
the function $c(h_0): (0, h_*]\rightarrow (c_r, c_*]$ is a bijection and we may write $h_0=h_0(c): (c_r, c_*]\rightarrow (0,h_*]$ as the inverse function. Note that 
\begin{equation}
	h_0'(c)>0,  \text{ for every } c\in (c_r, c_*). 
\end{equation}

\begin{lemma}
\label{lemma increasing in c}
	Let $h_*, r\in (0,1)$ and $c_*\in (0,e^{-\frac{1}{2}})$ be such that $h_*+r<1$ and $(h_*, h_*+r)$ is a good pair with respect to $c_*$. The integral
	\begin{equation}
		\Theta(c)=\Theta(c, h_0(c), r),
	\end{equation}
	is increasing on $(c_r, c_*]$.
\end{lemma}
\begin{proof}
	Set
\begin{equation*}
\begin{aligned}
f(c) =-(tr+h_0(c))^2-2\log(e^{-\frac{h^2_0(c)}{2}}-ctr).
\end{aligned}
\end{equation*}
By \eqref{eq 3.23.3}, it suffices to show that for every fixed $t\in (0,1)$, the function $f$ is increasing in $c$. 
We have
\begin{equation}
	f'(c) =\frac{2rt}{e^{-\frac{h^2_0(c)}{2}}-crt}[1- h_0'(c)(e^{-\frac{h^2_0(c)}{2}}-crt-ch_0(c))].
\end{equation}
Note that by \eqref{eq 3.24.6} and \eqref{eq 3.24.7},
\begin{equation}
c'(h_0)=e^{-\frac{h^2_0(c)}{2}}-cr-ch_0(c)\leq e^{-\frac{h^2_0(c)}{2}}-crt-ch_0(c),
\end{equation}
for each $t\in (0,1)$. 

Thus,
\begin{equation}
f'(c) \leq\frac{2rt}{e^{-\frac{h^2_0(c)}{2}}-crt}[1-h_0'(c)c'(h_0)]=0.
\end{equation}
\end{proof}

Finally, we are ready to show the desired estimate for $\Theta(c,h_0,r)$ when $(h_0, h_0+r)$ is a good pair with respect to $c$. 

\begin{lemma}
\label{lemma final estimate}
	Let $c\in (0, e^{-\frac{1}{2}})$. If $h_0, r>0$ are such that $(h_0, h_0+r)$ is a good pair with respect to $c$, then
	\begin{equation}
		\Theta(c,h_0, r)>\pi.
	\end{equation}
\end{lemma}

\begin{proof}
	Note that by definition of good pair, we have $h_0<m_1(c)<1$. By Lemma \ref{lemma increasing in r}, we have
	\begin{equation}
	\label{eq 3.24.9a}
		\Theta(c,h_0, r)\geq \Theta(c(\mathfrak r),h_0,\mathfrak r),
	\end{equation}
	for each $\mathfrak r\in (0,r)$. Let $c'=e^{-\frac{h_0^2}{2}}h_0$. Since $\lim_{\mathfrak r\rightarrow 0}c_{\mathfrak r}=0$ where $c_{\mathfrak r}$ is given in \eqref{eq 3.24.9}, there exists $\delta_0>0$ such that $c_{\mathfrak r}<c'$ for every $\mathfrak r\in (0, \delta_0)$. We also require that $\delta_0>0$ is sufficiently small so that $h_0+\delta_0<1$. 
	
	Since $(h_0, h_0+\mathfrak r)$ is a good pair with respect to $c(\mathfrak r)$, we have $0<\phi'_{c(\mathfrak r)}(h_0)=c(\mathfrak r)-h_0e^{-\frac{h_0^2}{2}}=c(\mathfrak r)-c'$. Therefore $c(\mathfrak r)>c'$. By Lemma \ref{lemma increasing in c},  for every fixed $\mathfrak r\in (0, \delta_0)$, we have
	\begin{equation}
	\label{eq 3.24.10}
		\Theta(c(\mathfrak r),h_0,\mathfrak r)\geq \Theta(c',h_0(c'),\mathfrak r).
	\end{equation} 
	
	Combining \eqref{eq 3.24.9a} and \eqref{eq 3.24.10}, we have
	\begin{equation}
		\Theta(c,h_0, r)\geq \Theta(c',h_0(c'),\mathfrak r)
	\end{equation}
	for every $\mathfrak r\in (0, \delta_0)$. Note that since $c'$ is fixed, $h_0(c')$ depends on $\mathfrak r$ as $\mathfrak r$ varies. Therefore, by Lemma \ref{inflim},
	\begin{equation}
		\Theta(c,h_0, r)\geq\liminf_{\mathfrak r\rightarrow 0} \Theta(c', h_0(\mathfrak r), \mathfrak r)\geq \frac{\pi}{\sqrt{1-m_1^2(c')}}>\pi.
	\end{equation}
\end{proof}

Theorem \ref{main uniqueness theorem} now follows directly from Lemmas \ref{lemma 3.23.1} and \ref{lemma final estimate}, and Proposition \ref{prop constant solutions}.

\section{Existence of symmetric solutions to the planar Gaussian Minkowski problem}

This section is dedicated to solving the planar, even Gaussian Minkowski problem in dimension 2, in the smooth setting. Suppose $\alpha\in (0,1)$ and $f\in C^{2,\alpha}(\s)$ is a positive even function, we will solve the equation
\begin{equation}
\label{eq 2.27.1}
	\frac{1}{2\pi}e^{-\frac{h'^2+h^2}{2}} (h''+h)=f,
\end{equation}
on $\s$. In fact, when combining with the existence result shown in \cite{MR4252759}, we show that \eqref{eq 2.27.1} has at least two solutions.

The following $C^0$ \emph{a priori} estimate is critically needed.
\begin{lemma}
\label{lemma C0}
	Suppose $f:\s\rightarrow \mathbb{R}$ is an even  positive function and $h=h_K\in C^2(\s)$, for some origin-symmetric convex body $K$ in $\mathbb{R}^2$, is an even solution to \eqref{eq 2.27.1}. If there exists $\tau>0$ such that
	\begin{equation}
		1/\tau<f<\tau,
	\end{equation}
	then there exists $\tau'>0$, dependent only on $\tau$, such that
	\begin{equation}
		\frac{1}{\tau'}<h<\tau'.
	\end{equation}
\end{lemma}
\begin{proof}
	We first show $h$ is bounded from above. Assume that $h$ achieves its maximum at $v_0\in \s$ and $h(v_0)=h_{\max}$. Evaluate \eqref{eq 2.27.1} gives us
	\begin{equation}
		\frac{1}{2\pi} e^{-\frac{{h_{\max}}^2}{2}}h_{\max}\geq f(v_0)>\frac{1}{\tau}.
	\end{equation}
	Note that the function $\frac{1}{2\pi} e^{-\frac{t^2}{2}}t$ goes to $0$ as $t\rightarrow \infty$. Therefore, there exists $\tau_1>0$ such that 
	\begin{equation}
	\label{eq 2.27.4}
		h_{\max}<\tau_1.
	\end{equation} 
	
	We now show $h_{\max}$ is also bounded from below. Observe that on $\s$, we have
	\begin{equation*}
		(h''+h)\geq \frac{1}{2\pi} e^{-\frac{h'^2+h^2}{2}} (h''+h)=f>\frac{1}{\tau}. 
	\end{equation*}
	Note also that the total integral of $h''+h$ over $\s$ is the perimeter of the convex body $K$ (whose support function is $h$). Therefore, we have
	\begin{equation*}
		\mathcal{H}^{1}(\partial K)> \frac{2\pi}{\tau}>0. 
	\end{equation*}
	On the other side, since perimeter is a monotone functional on the set of convex bodies (with respect to set inclusion), we have
	\begin{equation*}
		\mathcal{H}^{1}(\partial K)\leq 2\pi h_{\max}.
	\end{equation*}
	Combining the above two inequalities, we can find some $\tau_2>0$, such that 
	\begin{equation*}
		h_{\max}>\tau_2.
	\end{equation*}
	
	Finally, we show $h$ is bounded from below. Assume that $h$ achieves its minimum at $u_0\in \s$ and $h(u_0)=h_{\min}$. Note that by \eqref{eq 2.27.1}, we have
	\begin{equation*}
		\frac{1}{2}h(h''+h)={\pi} he^{\frac{h'^2+h^2}{2}}f\geq {\pi}hf>{\pi}h/\tau,
	\end{equation*}
	where we used the fact $h$ is nonnegative, which follows from the fact that $h$ is an even function (or, equivalently, $K$ is origin-symmetric). Observe the total integral of $\frac{1}{2}h(h''+h)$ on $\s$ is the area of $K$. Therefore, we have
	\begin{equation*}
		\mathcal{H}^2(K)>\frac{\pi}{\tau}\int_{\mathbb{S}^1}hdv.
	\end{equation*}
	By definition of support function, we have
	\begin{equation*}
		h(v)\geq h_{\max}|v\cdot v_0|.
	\end{equation*}
	As a consequence, there exists $\tau_3>0$ such that
	\begin{equation}
	\label{eq 2.27.2}
		\mathcal{H}^2(K)>\frac{\pi}{\tau}h_{\max} \int_{\mathbb{S}^1}|v\cdot v_0|dv=\tau_3h_{\max}>\tau_3\tau_2.
	\end{equation}
	Note that on the other hand, 
	\begin{equation*}
		K\subset (h_{\max} B_1) \cap \{x\in \mathbb{R}^2: |x\cdot u_0|\leq h_{\min}\},
	\end{equation*}
	which implies 
	\begin{equation}
	\label{eq 2.27.3}
		\mathcal{H}^2(K)\leq 4 h_{\max}h_{\min}< 4\tau_1h_{\min}.
	\end{equation}
	Combining \eqref{eq 2.27.2} and \eqref{eq 2.27.3} immediately gives $\tau_4>0$ such that 
	\begin{equation}
	\label{eq 2.27.5}
		h_{\min}>\tau_4.
	\end{equation}
	The existence of $\tau'>0$ now readily follows from \eqref{eq 2.27.4} and \eqref{eq 2.27.5}.
\end{proof}

Once we obtain the critical $C^0$ estimate, higher order estimates follow in the same way as in \cite{MR4252759}. Note that in \cite{MR4252759}, the higher order estimates (\cite[Lemma 6.5]{MR4252759}) only depends on the $C^0$ estimate (\cite[Lemma 6.4]{MR4252759}). We therefore state the following higher order estimates without duplicating the same proof as presented in \cite{MR4252759}.

\begin{lemma}[\emph{a priori} estimates]
\label{lemma higher order}
	Let $0<\alpha<1$. Suppose $f\in C^{2,\alpha}(\s)$ is an even function and there exists $\tau>0$ such that $\frac{1}{\tau}<f<\tau$ and $\|f\|_{C^{2,\alpha}}<\tau$. If the support function of $K\in \mathcal{K}_e^n$ is $C^{4,\alpha}$ and satisfies
	\begin{equation*}
		\frac{1}{{2\pi}} e^{-\frac{h'^2+h^2}{2}}(h''+h)=f,
	\end{equation*}
	then there exists $\tau'>0$ dependent only on $\tau$ such that
	\begin{enumerate}
		\item $\frac{1}{\tau'}<\sqrt{h'^2+h^2}<\tau'$
		\item $\frac{1}{\tau'}<h''+h<\tau'$
		\item $\|h\|_{C^{4,\alpha}}<\tau'$.
	\end{enumerate}
\end{lemma}

We are now ready to state the main existence result. 
\begin{theorem}[Existence of smooth, small solutions] \label{thm 2.28.6}
Let $0<\alpha<1$ and $f\in C^{2,\alpha}(\s)$ be a positive even function with $\|f\|_{L^1}<\frac{1}{\sqrt{2\pi}}$. Then, there exists a $C^{4,\alpha}$, origin-symmetric $K$ with $\gamma_2(K)<\frac{1}{2}$ such that its support function $h$ solves
\begin{equation}
\label{eq 2.28.1}
	\frac{1}{{2\pi}} e^{-\frac{h'^2+h^2}{2}}(h''+h)=f.
\end{equation}
\end{theorem}
\begin{proof}
	We prove the existence of solution using the degree theory for second-order nonlinear elliptic operators developed in Li \cite{MR1026774}.
	
	By Theorem \ref{main uniqueness theorem}, for sufficiently small $c_0>0$, the equation
	\begin{equation}
		\frac{1}{{2\pi}} e^{-\frac{h'^2+h^2}{2}}(h''+h)=c_0
	\end{equation}
	admits two constant solutions. Let $h_1\equiv r_1>0$ and $h_2\equiv r_2>0$ be the two constant solutions, with $r_1>r_2$. Then, for $i=1, 2$, we have
	\begin{equation}
		\frac{1}{{2\pi}} e^{-\frac{r_i^2}{2}}r_i =c_0.
	\end{equation}
	A quick analysis of the function $e^{-t^2/2}t$ yields that when $c_0$ is sufficiently small, we have $\gamma_2(r_1 B)>\frac{1}{2}$ and $\gamma_2(r_2B)<\frac{1}{2}$. We also require that $c_0>0$ is small enough so that $\|c_0\|_{L^1}<\frac{1}{\sqrt{2\pi}}$. We also require that $c_0>0$ is chosen so that the operator $L\phi=\phi''+(1-r_2^2)\phi$ is invertible.
	
	Let $F(\cdot; t):C^{4,\alpha}(\s)\rightarrow C^{2,\alpha}(\s)$ be defined as 
	\begin{equation}
		F(h;t) = h''+h-2\pi e^{\frac{h'^2+h^2}{2}}f_t,
	\end{equation} 
	for $t\in [0,1]$, where
	\begin{equation}
		f_t = (1-t)c_0+tf.
	\end{equation}
	Note that since $f>0$ is $C^{2,\alpha}$, there exists $\tau>0$ such that $\frac{1}{\tau}<f_t<\tau$ and $\|f_t\|_{C^{2,\alpha}}<\tau$. It is also simple to see $\|f_t\|_{L^1}<\frac{1}{\sqrt{2\pi}}$. We choose $\tau'>0$ according to  Lemmas \ref{lemma C0} and \ref{lemma higher order}.	
	Define $O\subset C^{4,\alpha}(\s)$ by
	\begin{equation}
		O=\left\{h\in C^{4,\alpha}(\s) \text{ is even}: \frac{1}{\tau'}<h<\tau', \frac{1}{\tau'}<h''+h<\tau', \|h\|_{C^{4,\alpha}}<\tau', \gamma_2(h)<\frac{1}{2}\right\}.
	\end{equation}
	Here $\gamma_2(h)=\gamma_2(K)$ where $K$ is the origin-symmetric convex body whose support function is $h$. This can be done since $h\in O$ is strictly convex. Note that $h_2\equiv r_2 \in O$, while $h_1\equiv r_1\notin O$.
	
	We claim now that for each $t\in [0,1]$, if $h\in \partial O$, then 
	\begin{equation}
		F(h;t)\neq 0.
	\end{equation}
	Indeed, if $F(h;t)=0$, then by Lemmas \ref{lemma C0} and \ref{lemma higher order}, it must be the case that $\gamma_2(h)=\frac{1}{2}$. However, by Gaussian isoperimetric inequality, this implies that $|S_{\gamma_2, K}|\geq \frac{1}{\sqrt{2\pi}}$. This is a contradiction to the fact that $F(h;t)=0$ and that $\|f_t\|<\frac{1}{\sqrt{2\pi}}$. 
	
	As a consequence, the degree of the map $F(\cdot, t)$ is well defined on $O$. Moreover, by Proposition 2.2 in Li \cite{MR1026774},
	\begin{equation}
	\label{eq 2.28.3}
		\deg (F(\cdot; 0), O, 0) = \deg(F(\cdot; 1), O, 0).
	\end{equation}
	Let us now compute $\deg (F(\cdot; 0), O, 0)$. For simplicity, write $F(\cdot)=F(\cdot; 0)$. Recall that $f=c_0$ is so chosen that $h\equiv r_2$ is the only solution in $O$ to \eqref{eq 2.28.1}. It is simple to compute the linearized operator of $F$ at the constant function $r_2$:
	\begin{equation}
		L_{r_2}\phi = \phi''+(1-r_2^2)\phi,
	\end{equation}
	which is invertible by our choice of $c_0$. By Proposition 2.3 in Li \cite{MR1026774}, this implies
	\begin{equation}
	\label{eq 2.28.4}
		\deg(F, O, 0)=\deg(L_{r_2},O,0)\neq 0,
	\end{equation}
	where the last inequality follows from Proposition 2.4\footnote{Proposition 2.4 in Li \cite{MR1026774} contains some typos, which were corrected by Li on his personal webpage.} in Li \cite{MR1026774}. Equations \eqref{eq 2.28.3} and \eqref{eq 2.28.4} now immediately implies that $\deg(F(\cdot; 1), O, 0)\neq 0$, which in turn implies the existence of solution.
\end{proof}

We remark that through a simple approximation argument, the regularity assumption on $f$ may be dropped.

\begin{theorem}
\label{thm 2.28.8}
	Let $f\in L^1(\s)$ be an even function such that $\|f\|_{L^1} <\frac{1}{\sqrt{2\pi}}$. If there exists $\tau>0$ such that $\frac{1}{\tau}<f<\tau$ almost everywhere on $\s$, then there exists an origin-symmetric $K$ with $\gamma_2(K)<\frac{1}{2}$ such that
	\begin{equation}
		dS_{\gamma_2, K}(v) = f(v)dv.
	\end{equation}
\end{theorem}
\begin{proof}
	We may approximate $f$ by a sequence of smooth functions $f_i$ such that $\mu_i=f_i dv$ converges weakly to $\mu=fdv$, and $\frac{1}{\tau}<f_i<\tau$ on $\s$. Because of weak convergence, we may (by discarding the first finitely many terms) assume $|\mu_i|<\frac{1}{\sqrt{2\pi}}$.
	
	By Theorem \ref{thm 2.28.6}, there exists $C^{4,\alpha}$, origin-symmetric $K_i$ with $\gamma_2(K_i)<\frac{1}{2}$ such that $S_{\gamma_2, K_i}=\mu_i$, or, equivalently, their support functions $h_i$ solves the equation
	\begin{equation}
		\frac{1}{{2\pi}} e^{-\frac{h_i'^2+h_i^2}{2}}(h_i''+h_i)=f_i
	\end{equation}
	By Lemma \ref{lemma C0}, there exists $\tau'>0$, independent of $i$, such that
	\begin{equation}
	\label{eq 2.28.7}
		\frac{1}{\tau'}B\subset K_i\subset \tau'B.
	\end{equation}
	Using Blaschke's selection theorem, we may assume (by possibly taking a subsequence) that $K_i$ converges in Hausdorff metric to an origin-symmetric convex body $K$. By \eqref{eq 2.28.7},
	\begin{equation}
		\frac{1}{\tau'}B\subset K\subset \tau'B.
	\end{equation} 
	The weak continuity of $S_{\gamma_2,K}$ in $K$ now implies that
	\begin{equation}
		S_{\gamma_2, K}=\mu.
	\end{equation}
	Note that by continuity of $\gamma_2$, we have $\gamma_2(K)\leq \frac{1}{2}$. That the inequality is strict follows from the Gaussian isoperimetric inequality and the fact that $\|f\|_{L^1}<\frac{1}{\sqrt{2\pi}}$.
\end{proof}

We remark that the \emph{a priori} estimates, Lemmas \ref{lemma C0} and \ref{lemma higher order}, work in higher dimensions with no essential change in the proof. Therefore, the ability to extend the existence results---Theorems \ref{thm 2.28.6}, \ref{thm 2.28.8}---depends on our ability to establish uniqueness result for constant $f$ (Theorem \ref{main uniqueness theorem}) to higher dimensions.

\begin{conjecture}
	Let $n\geq 3$. If $h$ is a nonnegative solution to the equation
	\begin{equation}
		\frac{1}{(\sqrt{2\pi})^n} e^{-\frac{|\nabla h|^2+h^2}{2}}\det (\nabla^2 h+hI)=c>0, 
	\end{equation}
	then $h$ must be a constant solution.
\end{conjecture}

\end{document}